\documentclass[11pt]{article}

\usepackage[english]{babel}

\usepackage[letterpaper,top=2cm,bottom=2cm,left=3cm,right=3cm,marginparwidth=1.75cm]{geometry}

\usepackage[utf8]{inputenc}
\usepackage{amsmath}
\usepackage{amsthm}
\usepackage{multicol}
\usepackage{amssymb}
\usepackage{algpseudocode} 
\usepackage{algorithm}
\usepackage[mathcal]{euscript}
\usepackage[english]{babel}
\usepackage{adjustbox}
\usepackage[T1]{fontenc}
\usepackage{url}            
\usepackage{booktabs}       
\usepackage{amsfonts}       
\usepackage{nicefrac}       
\usepackage{microtype}      
\usepackage{tikz}
\usetikzlibrary{positioning}
\usepackage{graphicx}
\usepackage[nottoc,notlot,notlof]{tocbibind}
\usepackage{caption}
\usepackage{subcaption}
\usepackage{mathabx}
\usepackage[inline]{enumitem}
\usepackage{hyperref}

\title{How to measure multidimensional variation?}

\author{
    Gennaro Auricchio\thanks{gennaro.auricchio@unipd.it} \\ 
    Department of Mathematics,\\ University of Padua \\
    \and
    Paolo Giudici \thanks{paolo.giudici@unipv.it} \\
    Department of Economics,\\
    University of Pavia
    \and
    Giuseppe Toscani\thanks{giuseppe.toscani@unipv.it} \\
    Department of Mathematics,\\ University of Pavia \\
}

\newtheorem{theorem}{Theorem}

\newtheorem{remark}{Remark}

\newtheorem{definition}{Definition}

\def\bx{{\bf x}}
\def\bX{{\bf X}}
\def\bY{{\bf Y}}

\def\by{{\bf y}}
\def\bm{{\bf m}}

\def\R{\mathbb{R}}

\def \PP{\mathcal{P}}

\def \erre{\mathbb{R}}

\def \enne{\mathbb{N}}

\def\bx{{\bf x}}
\def\bX{{\bf X}}
\def\bK{{\bf K}}
\def\bR{{\bf R}}
\def\bY{{\bf Y}}

\def\by{{\bf y}}
\def\bm{{\bf m}}

\def\bc{{\bf c}}

\def\R{\mathbb{R}}

\newtheorem{corollary}{Corollary}

\date{}

\begin{document}



\maketitle

\begin{abstract}
The coefficient of variation, which measures the variability of a distribution from its mean, is not uniquely defined in the multidimensional case, and so is the multidimensional Gini index, which measures the inequality of a distribution in terms of the mean differences among its observations.  
In this paper, we connect these two notions of sparsity, and propose a multidimensional coefficient of variation based on a multidimensional  Gini index. We demonstrate that the proposed coefficient  possesses the properties of the univariate coefficient of variation. We also show its connection with the Voinov-Nikulin coefficient of variation, and compare it  with the other  multivariate coefficients  available in the literature.
\end{abstract}

\vskip.2cm


\section{Introduction}
\label{sec:intro}

The coefficient of variation is the standard measure to summarize through a scalar value the variability of a set of points in a statistical distribution.
For one dimensional distributions the universally accepted definition of the coefficient of variation (CV) is: 
\begin{equation}
    \label{eq:std_CV}
    CV(\mu)=\frac{\sigma}{|m|},
\end{equation}
where $\sigma$ and $m$ are the standard deviation and the mean, respectively, of a statistical distribution $\mu$.
Albeit the coefficient in \eqref{eq:std_CV} has been used for more than a century to handle one dimensional data, there is still no universally accepted way to measure the variability of a multidimensional distribution.
Indeed, different definitions of a multivariate coefficient of variation are present in the statistical literature \cite{aerts2015multivariate}. 
Among them, only a few  enjoy the most basic property that the univariate coefficient of variation \eqref{eq:std_CV} possesses, such as being scale-invariant \cite{auricchio2024extending}.
A similar issue occurs for another measure of sparsity, the Gini index. 
The (univariate) Gini index is a measure of sparsity that summarises with a scalar value the inequality of a statistical distribution in terms of the mean difference between its values, normalised by their mean. 
Given a distribution $\mu$ supported over $[0,\infty)$ with mean $m$, the Gini index \cite{gini1914sulla,gini1921measurement} is defined as
\begin{equation}
\label{eq:Gini_intro}
    G(\mu):=\frac{1}{2|m|}\int_{\erre}\int_{\erre}|x-y|\mu(dx)\mu(dy).   
\end{equation}
On the one hand, the Gini index is a measure of sparsity, thus easier to interpret than the coefficient of variation: $G(\mu)$ is always between $0$ and $1$, where $0$ represents a status of perfect equality, while $1$ represent the maximum inequality.
On the other hand, the coefficient of variation is well-defined for every probability measure that has non-null mean, whereas the Gini index requires $\mu$ to be supported only on the positive half-line in order to attain a value between $0$ and $1$. 
Despite the differences, the coefficient of variation in \eqref{eq:std_CV} and the Gini index in \eqref{eq:Gini_intro} have many common traits, as they both aim to summarise the information about the sparsity of a distribution through a scalar value that does not depend on the unit of measurement.
These similarities become apparent when we consider a Gaussian distribution.
Let $\mu=\mathcal{N}(m,\sigma)$ be a Gaussian distribution with both $m,\sigma>0$. 
Then, through a simple computation, we have that
\[
G(\mu) = \frac{1}{2|m|}\int_{\erre}\int_{\erre}|x-y|\mu(dx)\mu(dy)=\frac 2{\sqrt\pi}\frac{\sigma}{|m|}.
\]
Thus, for a Gaussian distribution, the Gini index is proportional to the coefficient of variation.
Furthermore, if we consider an alternative ``squared'' definition of the Gini index:
\begin{equation}
\label{eq:G2_intro}
    G_2(\mu):=\Big(\frac{1}{2m^2}\int_{\erre}\int_{\erre}|x-y|^2\mu(dx)\mu(dy)\Big)^{\frac{1}{2}},
\end{equation}
based on the $L^2$ norm rather than the $L^1$ norm as in \eqref{eq:Gini_intro}, we have that
\begin{align*}
    G_2(\mu)&=\Big(\frac{1}{2m^2}\int_{\erre}\int_{\erre}(x^2-2xy+y^2)\mu(dx)\mu(dy)\Big)^{\frac{1}{2}}\\
    &=\frac{1}{\sqrt{2}|m|}\Big(\int_{\erre} x^2\mu(dx)-2m^2+\int_{\erre}y^2\mu(dy)\Big)^{\frac{1}{2}}\\
    &=\frac{1}{\sqrt{2}|m|}\Big(2\sigma^2\Big)^{\frac{1}{2}}=\frac{\sigma}{|m|},
\end{align*}
which shows that, for any distribution $\mu$, the squared Gini index $G_2$ coincides exactly with the coefficient of variation.
The above findings indicate that the Gini index $G_1$ and the coefficient of variation $CV$ are the $L^1$ and $L^2$ norms of the same function $\frac{|x-y|}{|m|}$.

\subsection*{Related Work}

The univariate coefficient of variation is a well-established measure of sparsity that has been widely adopted by the scientific community.
Despite its widespread use, a universally accepted method to extend this coefficient to distributions supported over high-dimensional spaces has yet to be developed.
This lack of clarity on how a multivariate coefficient of variation (MCV) should be defined and what properties it should enjoy lead to multiple extensions:
\begin{enumerate*}[label=(\roman*)]
    \item the Reyment's coefficient of variation \cite{reyment1960studies},
    \item the Van Valen's coefficient of variation \cite{van1974multivariate},
    \item the Albert and Zhang's coefficient of variation \cite{albert2010novel}, and
    \item the Voinov-Nikulin's coefficient of variation \cite{voinov2012unbiased}.
\end{enumerate*}
Although these coefficients have been recently unified under a general definition in \cite{colin2024towards}, they all exhibit different properties. 
For example, the Voinov-Nikulin's coefficient of variation is the only \textit{scale invariant} MCV, \textit{i.e.} it does not change if we rescale one or more of the components of the vectorial quantity under study by changing their respective units of measure \cite{aerts2015multivariate}.
The Gini index was firstly introduced by the Italian statistician Corrado Gini, \cite{gini1914sulla,gini1921measurement} along with the less famous index proposed by Gaetano Pietra \cite{Pie}.
The hunger for measures of inequality, especially in economics applications, is still alive, as witnessed by the evergrowing amount of works in this direction \cite{giudici2023safe,betti2008advances,coulter2019measuring,babaei2025rank,hao2010assessing}. 
We refer the reader to \cite{banerjee2020inequality,eliazar2018tour,eliazar2020gini} for an exhaustive review of this topic.
Similarly to the coefficient of variation, there is no universally accepted extension of the Gini index to the multivariate case and, for this reason, different higher-dimensional extensions have been proposed over time.
The earliest approach relies on differential geometry methods and was developed by Taguchi \cite{taguchi1972one, taguchi1972two}. 
Subsequent contributions were made by Arnold \cite{arnold2008pareto}, Arnold and Sarabia \cite{arnold2018analytic}, Gajdos and Weymark \cite{gajdos2005multidimensional}, Koshevoy and Mosler \cite{koshevoy1996lorenz,koshevoy1997multivariate}, and Sarabia and Jorda \cite{sarabia2020lorenz}. 
Unfortunately, as discussed in \cite{arnold2018analytic}, these multivariate extensions are largely guided by elegant mathematical frameworks but often lack practical applicability and interpretability.
A recent line of research defines an higher dimensional Gini index based on the principal components of the data, rather than on the data itself \cite{toscani2022fourier}.
This approach has lead to the definition of new measures of inequality that are scale invariant \cite{toscani2024} as well as to a natural way to express the multivariate Gini index as a suitable convex combination of the Gini indexes of its principal components \cite{auricchio2024extending}.
This has been then fruitfully used to define discrepancies which are stable under change of scale \cite{auricchio2024multivariate}.

\subsection*{Our Contribution and Structure of the Paper}

Building on the above references, and on a recent generalization of the Gini index for multivariate distributions \cite{auricchio2024extending,toscani2024}, in this paper we shed further light on the connections between the Gini index and the coefficient of variation, for the multidimensional case.
Through this connection, that is based on a solid theoretical foundation \cite{auricchio2024extending, auricchio2024multivariate}, we compare the existing multivariate coefficients of variation in terms of their properties.

The paper is organised as follows. 
After some preliminary background in Section \ref{sec:pre}, in Section \ref{sec:properties} we define the  properties that a multivariate coefficient of variation should have, in analogy of those of the univariate coefficient \eqref{eq:std_CV}. 
In Section \ref{sec:ourCV}, we present a Gini based multivariate coefficient of variation, and demonstrate its properties, along with those of the the Voinov-Nikulin's MCV, to which it is strictly related. 
In Section \ref{sec:other}, we compare the other existing MCVs in terms of their properties.  
Section \ref{sec:6} presents some examples and simulation studies which help clarifying the importance of our proposal, and concludes with some final remarks.

\section{Preliminaries}
\label{sec:pre}

In what follows, we denote with $|\bx|=\sqrt{x_1^2+x_2^2+\dots+x_n^2}$ the Euclidean norm of $\bx\in\R^n$.
Let $q \ge 1$, and let  $\PP_q(\R^n)$ denote the class of all probability measures $\mu$ on $\R^n$ with finite moment of order $q$, that is
\begin{equation}
    \label{eq:secondmoment}
    \int_{\R^n} |\bx|^q \mu(d\bx) < + \infty.
\end{equation}
Since we are working in Euclidean spaces, every probability measure $\mu$ is canonically associated with a random vector $\bX$.
Henceforth, we write $\bX\sim\mu$ to denote the random vector associated with the probability measure $\mu\in \PP_q(\R^n)$ and use $\bX$ and $\mu$ interchangeably.
In particular, given a random vector $\bX$ and a MCV $\gamma$, we denote the variability of $\bX$ according to $\gamma$ with either $\gamma(\bX)$ or $\gamma(\mu)$, depending on which expression is more convenient in the context.
For every $\mu\in\PP_2(\R^n)$, we have that the mean value vector  $\bm := \mathbb{E}(\bX)= \int_{\R^n} \bx \, \mu(d\bx)$ and the covariance matrix $\Sigma$, namely
\[
(\Sigma)_{ij} = cov(X_i,X_j) := \int_{\R^n} x_i \, x_j \, \mu(d\bx) - \left( \int_{\R^n} x_i \mu(d\bx) \right) \, \left( \int_{\R^n} x_j \mu(d\bx) \right),
\]
of a random vector ${\bf X}$, whose probability law is $\mu \in \PP_2(\R^n)$, are well defined. 
Similarly, the correlation matrix of $\mu$, namely $P$, is well-defined and we have that
\[
P_{i,j}=\frac{(\Sigma)_{i,j}}{\sqrt{Var(X_i)Var(X_j)}}.
\]
We remark that $\Sigma$ and $P$ are symmetric and positive semidefinite.
In what follows, we assume that $\Sigma$ is invertible and that $\bm\neq0$, unless we specify otherwise.

\subsection{Whitening processes}
\label{sec:WhiteProc}
Given a $n$-dimensional random vector $\bX$ whose mean is $\bm$ and covariance matrix is $\Sigma$, we have that the $n$-dimensional random vector $\bY=W\bX$  has mean $W\bm$ and $W^T\Sigma W$ for every $n\times n$ matrix $W$, \cite{mahalanobis2018generalized, li1998sphering}. 
We say that $W$ is a whitening matrix for a random vector $\bX$ if the covariance matrix of $W\bX$ is the identity matrix, that is $W^T\Sigma W=Id$ or, equivalently, if
\begin{equation}
\label{whi2}
W^TW =\Sigma^{-1}.
\end{equation}
Owing to the fact that the whitening matrix is not unique, there are a variety of whitening processes that are commonly used \cite{kessy2018optimal}. 
Indeed, it is easy to see that, given a whitening matrix $W$ for $\bX$ and an orthogonal matrix $Z$, the matrix $W' = ZW$ satisfies \eqref{whi2}.
However, as shown in \cite{auricchio2024extending},  only a few of these processes possess the \emph{scale stability} property, which ensures that the random vector obtained by whitening $\bX:=(X_1,X_2,\dots,X_n)$ or $\bX'=(X_1,X_2,\dots,aX_i,\dots,X_n)$ is the same for every $a>0$ and $i\in[n]$.
More formally, a whitening process is a map $\mathcal{S}$ that, given a random vector $\bX$ or its associated probability measure $\mu$, returns a whitening matrix $W_\mu$.
Denoted with $diag(\vec q)$ the diagonal matrix whose diagonal values are $\vec q =(q_1,q_2,\dots,q_n)$, we have that $\mathcal{S}$ is scale stable if, given a random vector $\bX$ and a diagonal matrix $Q=diag(q_1,q_2,\dots,q_n)$ such that $q_i > 0$, it holds
\[
    W_\mu\bX=W_{\mu^{Q}}(Q\bX),
\]
where $\mu^Q$ is the probability measure associated with the random vector $Q\bX$.
A \emph{scale stable} whitening processes exists; for example, in \cite{auricchio2024extending}, it has been shown that the Cholesky whitening and the \textit{Zero-phase Components Analysis} whitening transformation applied to the correlation matrix (ZCA-cor whitening) are both scale stable.
For the sake of simplicity, we focus only on the ZCA-cor whitening associated with a probability measure $\mu$, which is defined as
\begin{equation}
    \label{ZCA}
W_\mu^{ZCA} = P^{-1/2}V^{-1/2}=G^T\Theta GV^{-1/2},
\end{equation}
where 
\begin{enumerate*}[label=(\roman*)]
    \item $V$ is a diagonal matrix containing the variances of the components of $\bX$,
    \item $G$ is the orthonormal matrix induced by the eigenvectors of $P$, and
    \item $\Theta$ is the diagonal matrix containing the eigenvalues of $P$.
\end{enumerate*} 
Therefore, from now on, any whitening matrix is tacitly assumed to be the ZCA-cor whitening matrix associated to the random vector at hand, unless we specify otherwise.
Finally, given a random vector $\bX\sim\mu$, we denote with $W_{\mu}^{ZCA}\bX=:\bX^*\sim\mu^*$ its whitened counterpart and refer to the entries of $\bX^*$ as the \textit{principal components} of $\bX$.

\subsection{The Coefficient of Variation and its High-dimensional Generalizations}
\label{sec:otherCVs}
Given a one dimensional random variable with non null-mean, the standard way to measure its sparsity is through the coefficient of variation, defined as $CV(\mu):={\sigma}/{|m|}$, where $\sigma$ is the standard deviation of the random variable and $m\neq 0$ its mean.
There are different ways to extend the coefficient of variation to higher dimensional random vectors.
To the best of our knowledge, there are four extension of the coefficient of variation to the multivariate setting.
Given a probability measure $\mu\in \PP_2(\R^n)$, whose mean is $\bm$ and variance matrix is $\Sigma$, the known multivariate coefficients of variation are defined as follows
\begin{enumerate}
    \item Voinov-Nikulin's coefficient of variation \cite{voinov2012unbiased}, which is defined as
    \begin{equation}
        \label{eq:VN} \gamma_{VN}(\mu):=\sqrt{\frac{1}{\bm^T\Sigma^{-1}\bm}};
    \end{equation}
    \item Reyment's coefficient of variation \cite{reyment1960studies}, which is defined as
    \begin{equation}
        \label{eq:RCV} \gamma_{R}(\mu):=\sqrt{\frac{det(\Sigma)^{\frac{1}{n}}}{\bm^T\bm}};
    \end{equation}
    \item Van Valen's coefficient of variation \cite{van1974multivariate}, which is defined as
    \begin{equation}
        \label{eq:VVCV} \gamma_{VV}(\mu):=\sqrt{\frac{tr(\Sigma)}{\bm^T\bm}};
    \end{equation}
    \item Albert and Zhang's coefficient of variation \cite{albert2010novel}, which is defined as
    \begin{equation}
        \label{eq:AZCV} \gamma_{AZ}(\mu):=\sqrt{\frac{\bm^T\Sigma \bm}{\bm^T\bm}}.
    \end{equation}
\end{enumerate}

\section{Properties of Multivariate Coefficients of Variation}
\label{sec:properties}

In this section, we outline the properties that the classic univariate coefficient of variation 
\begin{equation}
\label{eq:CVDes}
    CV(\mu)=\frac{\sigma}{|m|}
\end{equation} satisfies and extend them to a higher dimensional setting, drawing inspiration from \cite{hurley2009comparing}.
%
%
We then employ these properties as guidelines to compare the four multivariate coefficients of variation introduced in the previous Section.
%


\subsection{Coherence}
The first property that a multivariate coefficient of variation should possess is that, when applied to a one dimensional measure, it coincides with the  univariate coefficient of variation in \eqref{eq:CVDes}.
We name this property coherence.

\begin{definition}
    A multivariate coefficient of variation $\gamma$ is coherent if
    \[
        \gamma(\mu)=\frac{\sigma}{|m|}
    \]
    for every $\mu\in \PP_2 (\erre)$ with mean $m\neq 0$ and variance $\sigma^2$.
\end{definition}

It is easy to show that all the four multivariate coefficients of variation introduced in Section \ref{sec:otherCVs} possess this property.

\subsection{Scale Invariancy}

One of the most important properties a multivariate coefficient of variation should possess is scale invariance.

\begin{definition}[Scale Invariance]
    Given a random $n$-dimensional vector $\bX\sim\mu$, with $\mu \in \PP_2(\R^n)$, and a $n\times n$ invertible matrix $A$, we denote with $\mu_A$ the probability measure associated with the random vector $\bY=A\bX$.
    A coefficient of variation $\gamma$ is scale invariant if $\gamma(\mu_A)=\gamma(\mu)$ for any invertible $n\times n$ matrix $A$.
\end{definition}

Scale invariance ensures that the sparsity of a measure depends only on the probability measure itself, and neither on the unit of measure nor on the coordinates we use to describe the space.
Notice that, if $\gamma$ is scale invariant, then the sparsity of any random vector $\bX$ is equal to the sparsity of its principal components $\bX^*$, that is $\gamma(\bX)=\gamma(\bX^*)$.

\subsection{Splitting Uncorrelated Features}

This property states that the sparsity of $n$ features that are not correlated can be expressed as a function of the sparsity of each entry of $\bX$.

\begin{definition}
    A multivariate coefficient of variation $\gamma$ is Splitting Uncorrelated Features (SUF) if, for every $n$, there exists a function $G:\erre^n\to[0,\infty)$ such that
    \begin{equation}
        \gamma(\bX)=G(\gamma(X_1),\dots,\gamma(X_n)),
    \end{equation}
    where $\bX=(X_1,\dots,X_n)$ is a random vector with uncorrelated entries, i.e. such that its covariance matrix $\Sigma$ is diagonal.
\end{definition}
Notice that a multivariate coefficient of variation satisfying this property enables the aggregation of sparsity information without the need of merging datasets.
Indeed, if we have two sets of uncorrelated information $\bX'$ and $\bX''$, we can compute $\gamma((\bX',\bX''))$ using only $\gamma(\bX')$ and $\gamma(\bX'')$.

\subsection{Rising Tide}

The rising tide property was firstly introduced in \cite{hurley2009comparing}.  
When applied to measures of inequality which are defined on positive quantities, the property states that if we add any positive constant to all elements of the vector, the value of the MCV decreases.
When considering measures of inequality, such as Gini index, the quantity to add is necessarily positive.
However, since coefficients of variation are defined for more general probability distributions  with non-null mean, we cannot add any constant to the random vector.
For example, if $\bX$ has mean $\bm\neq0$, the vector $\bX-\bm$ has null mean.
For this reason, we define below a  revised version of the rising tide property.

\begin{definition}
    \label{def:rising}
    Let $\gamma$ be a multivariate coefficient of variation.
    We say that $\gamma$ satisfies the rising tide property if, given $\bX\sim\mu$ an $n$-dimensional random vector with mean $\bm$ and covariance matrix $\Sigma$, we have that
    \[
        \gamma(\mu)\ge\gamma(\mu_{\bc})
    \]
    for every $\bc\in\erre^n$ such that $\bc^T\Sigma^{-1}\bm \ge 0$.
\end{definition}

\begin{remark}
    The condition $\bc^T\Sigma^{-1}\bm \ge0$ is necessary to ensure that the mean of the principal components of a random vector $\bX$ do not nullify.
    Given a whitening matrix $W$ associated to $\bX$, we have that the mean of $\bX^*=W\bX$ is $\bm^*=W\bm$.
    Therefore, when we add a constant value $\bc$ to $\bX$, we need to ensure that $\bm^*$ and the whitened vector $W\bc$ point in the same direction.
    This happens if and only if $(W\bc)^T W\bm \ge 0$ or equivalently, if $\bc^TW^TW\bm=\bc^T\Sigma^{-1}\bm\ge0$.
    Notice that, when $X$ is a random variable, the condition $\bc^T\Sigma^{-1}\bm \ge 0$ boils down to ${(cm)}/{\sigma^2}>0$, that is $cm>0$: $c\neq 0$ and $m$ must have the same sign.
\end{remark}

Consider the case in which $X\sim\mu$ is a random variable with positive mean.
Definition \ref{def:rising} then reads that if we add any positive constant to $X$ the coefficient of variation decreases.
Indeed, adding a constant to $X$ does not affect its variance matrix, hence we have that
\[
    CV(\mu_{c})=\frac{\sigma}{|m+c|}=\frac{\sigma}{m+c}<\frac{\sigma}{m}=\frac{\sigma}{|m|}=CV(\mu).
\]
Notice that this property does not hold if we allow $c$ to be negative.
This property describes the fact that two random vectors having the same covariance matrix but different means entail different levels of variability.

\subsection{Cloning}

This property states that the value of the multivariate coefficient of variation should not change if we consider a coupling $(\bX,\bX)$ composed of two independent copies of the vector $\bX$.
In other words, measuring the same quantities twice does not affect the sparsity of the measured quantities.

\begin{definition}
    A coefficient of variation $\gamma$ satisfies the cloning property if, for any given $\bX$, we have that
    \[
        \gamma(\bX)=\gamma\big((\bX,\bX)\big)
    \]
    where $(\bX,\bX)$ is an independent coupling.
\end{definition}

\subsection{Dimension Stability}

Lastly, we consider dimension stability.
This property describes the limiting behaviour of a multivariate coefficient of variation when evaluated on a sequence of independent random vectors.
It seems desirable that the multivariate coefficient of variation converges toward a constant value $L$ if the marginal measures $\mu_i$ are such that $CV(\mu_i)=\frac{\sigma_i}{|m_i|}$ exists, it is finite, and $\lim_{i\to\infty}\frac{\sigma_i}{|m_i|}=L$.
This property implies that if we have an increasing number of independent features whose coefficient of variation converge to $L$, the sparsity of the random vector that contains them converges toward the common value $L$ as the number of independent features increases.

\begin{definition}
\label{def:dimstable}
    Let $\gamma$ be a multivariate coefficient of variation and let $\{\mu^{(n)}\}_{n\in\mathbb{N}}$ be a sequence of independent probability measures such that
    \begin{itemize}
        \item for every $n\in\mathbb{N}$, we have that $\mu^{(n)}\in \PP_2(\erre^n)$,
        \item for every $n\in\mathbb{N}$, let $\bX^{(n)}=(X_1^{(n)},X_2^{(n)},\dots,X_n^{(n)})\sim\mu^{(n)}$, then each entry of $\bX^{(n)}$ is independent from the other entries, and
        \item the marginal of $\mu^{(n)}$ on the first $n-1$ coordinates is equal to $\mu^{(n-1)}$ or, equivalently $\bX^{(n-1)}=\bX_{-n}^{(n)}=(X_1^{(n)},X_2^{(n)},\dots,X_{n-1}^{(n)})$ for every $n$.
    \end{itemize}
    Denoted with $\bm^{(n)}$ the mean of $\mu^{(n)}$ and with $\sigma^{(n)}_{n,n}$ be the standard deviation of the $n$-th component of $\mu^{(n)}$, we say that $\gamma$ is \textit{dimension stable} if the limit of $\gamma(\mu^{(n)})$ exists whenever $\lim_{n\to\infty}\frac{\sigma_{n,n}^{(n)}}{|m^{(n)}_n|}$ exists and it holds
    \[
        \lim_{n\to\infty}\gamma(\mu^{(n)})=\lim_{n\to\infty}\frac{\sigma_{n,n}^{(n)}}{|m^{(n)}_n|}.
    \]
\end{definition}

In other words, a coefficient of variation is dimension stable if, when adding to a random vector a sequence of independent one dimensional random variables, whose coefficient of variation converges to a value $L$, the coefficient of variation of the resulting sequence of random vectors converge to $L$ as well.

\section{The Gini Index as a Coefficient of Variation}
\label{sec:ourCV}

Building on the connection established in the introduction, in this section we show that the squared multivariate Gini index extending \eqref{eq:G2_intro} yields a multivariate coefficient of variation with the properties discussed in the previous section.
Let $\mu \in \PP_2(\R^n)$ be a probability measure, let $\bm$ be its mean value and $\Sigma$ be its covariance matrix.
The squared Gini index of $\mu$ is then defined as follows
\begin{equation}
\label{eq:G_2_main}
 G_2(\mu):=\Bigg(\frac{1}{2\bm^T\Sigma^{-1}\bm}\int_{\erre^n}\int_{\erre^n}(\bx-\by)^T\Sigma^{-1}(\bx-\by)\mu(d\bx)\mu(d\by)\Bigg)^{\frac{1}{2}}.
\end{equation}
We then have the following.

\begin{theorem}
    \label{thm:mainequiv}
    Given any probability measure $\mu\in\PP_2(\erre^n)$, we have that
    \begin{equation}
        G_2(\mu)=\sqrt{n}\;\gamma_{VN}(\mu)=\sqrt{\frac{n}{\bm^T\Sigma^{-1}\bm}}.
    \end{equation}
\end{theorem}

\begin{proof}
    By definition, we have that
    \begin{align*}
         G_2^2(\mu)&=\frac{1}{2\bm^T\Sigma^{-1}\bm}\int_{\erre^n}\int_{\erre^n}(\bx-\by)^T\Sigma^{-1}(\bx-\by)\mu(d\bx)\mu(d\by)\\
         &=\frac{1}{2\bm^T\Sigma^{-1}\bm}\Big(\int_{\erre^n}\int_{\erre^n}(\bx^T\Sigma^{-1}\bx+\by^T\Sigma^{-1}\by-2\bx^T\Sigma^{-1}\by)\mu(d\bx)\mu(d\by)\Big)\\
         &=\frac{1}{\bm^T\Sigma^{-1}\bm}\Big(\int_{\erre^n}\int_{\erre^n}(\bx^T\Sigma^{-1}\bx-\bx^T\Sigma^{-1}\by)\mu(d\bx)\mu(d\by)\Big)\\
         &=\frac{1}{\bm^T\Sigma^{-1}\bm}\Big(\int_{\erre^n}(\bx^T\Sigma^{-1}\bx)\mu(d\bx)-\bm^T\Sigma^{-1}\bm\Big)\\
         &=\frac{1}{\bm^T\Sigma^{-1}\bm}\Big(\int_{\erre^n}(\bx-\bm)^T\Sigma^{-1}(\bx-\bm)\mu(d\bx)\Big)
    \end{align*}
    since $\Sigma^{-1}$ is symmetric.
    We then rewrite $G_2^2(\mu)$ as it follows
    \begin{equation}
        G_2^2(\mu)=\frac{1}{\bm^T\Sigma^{-1}\bm}\int_{\erre^n}(\Sigma^{-\frac{1}{2}}(\bx-\bm))^T(\Sigma^{-\frac{1}{2}}(\bx-\bm))\mu(d\bx)\mu(d\by);
    \end{equation}
    if we change variable and set $\by^*=\Sigma^{-\frac{1}{2}}(\bx-\bm)$, we have that
    \begin{equation}
        G_2^2(\mu)=\frac{1}{\bm^T\Sigma^{-1}\bm}\int_{\erre^n}||\by^*||_2^2\mu^*(d\by^*),
    \end{equation}
    where $\mu^*$ is a probability measure whose marginals have null mean and unitary variance.
    To conclude, we notice that
    \begin{align*}
        G_2^2(\mu)&=\frac{1}{\bm^T\Sigma^{-1}\bm}\int_{\erre^n}||\by^*||_2^2\mu^*(d\by^*)=\frac{1}{\bm^T\Sigma^{-1}\bm}\int_{\erre^n}\sum_{i=1}^n |y^*_i|^2\mu^*(d\by^*)\\
        &=\frac{1}{\bm^T\Sigma^{-1}\bm}\sum_{i=1}^n\int_{\erre^n} |y^*_i|^2\mu^*(d\by^*)=\frac{1}{\bm^T\Sigma^{-1}\bm}\sum_{i=1}^n1=\frac{n}{\bm^T\Sigma^{-1}\bm},
    \end{align*}
    since each marginal of $\mu^*$ has unitary variance and null mean.
\end{proof}

Theorem \ref{thm:mainequiv} has two important implications.
First of all, it builds on the identity presented in the Section \ref{sec:intro} between the squared Gini index and the coefficient of variation, extending it to the  multivariate case. 
In other words, the multidimensional squared Gini index is both a measure of variability from the mean and a measure of inequality. 
We will shortly see that $G_2$ possesses all properties defined in  Section \ref{sec:properties}.
Second, it represents the squared Gini index as  proportional to the Voinov-Nikulin's coefficient of variation, with a correction term equal to $\sqrt{n}$. 
%
%
This simple observation points out that, among the several options, the Voinov-Nikulim's coefficient of variation is the rightful multidimensional extension of the classic coefficient of variation, albeit corrected by a multiplicative constant.

\subsection{Properties of \texorpdfstring{$G_2$}{G2}}
In what follows, we show that the $G_2$ multivariate coefficient of variation possesses all the properties highlighted in Section \ref{sec:properties}.

\subsubsection{Coherence}
The coherence of $G_2$ follows from the fact that, for $n=1$, we have that $\Sigma^{-1}=\frac{1}{\sigma^2}$, thus
\[
G_2(\mu)=\sqrt{\frac{1}{\frac{m^2}{\sigma^2}}}=\frac{\sigma}{|m|}.
\]
Notice that, since for $n=1$ $G_2(\mu)=\gamma_{VN}(\mu)$ for every probability distribution $\mu\in \PP_2(\erre)$, also the Voinov-Nikulin's coefficient of variation is coherent.
\subsubsection{Scale Invariance}
Owing to Theorem \ref{thm:mainequiv} and the scale invariancy of the Voinov-Nikulin's coefficient of variation \cite{aerts2015multivariate}, we have that $G_2$ is scale invariant as well.

\subsubsection{Splitting Uncorrelated Features}
We now show that $G_2$ satisfies the SUF property by providing a function $G$ that allows us to express $G_2(\bX)$ as a function of the coefficients of variation of the entries of $X_i$ when each $X_i$ is uncorrelated from the other entries.
\begin{theorem}
    Let $\bX\sim\mu\in\PP_2(\erre^n)$ be a vector with uncorrelated entries. 
    Denoted with $\Sigma$ the diagonal matrix describing the covariance of $\bX$, we have that $G_2(\bX)=G(CV(X_1),\dots,CV(X_n))$, where
    \begin{equation}
    G(y_1,y_2,\dots,y_n)=\sqrt{\frac{n}{y_1^{-2}+\dots+y_n^{-2}}}.
\end{equation}
In particular, we have that $G_2(\bX)$ is the harmonic mean of the coefficients of variation of the entries of $\bX$. 
\end{theorem}

\begin{proof}
  From a simple computation, we have that
\begin{align*}
    G(CV(X_1),CV(X_2),\dots,CV(X_n))&=\sqrt{\frac{n}{\frac{m_1^2}{\sigma_1^2}+\dots+\frac{m_n^2}{\sigma_n^2}}}=\sqrt{\frac{n}{\bm^T\Sigma^{-1}\bm}}=G_2(\bX),
\end{align*}
which concludes the proof.
\end{proof}

It is easy to adapt the argument to suit $\gamma_{VN}$.
In particular, we have that $\gamma_{VN}$ is an unnormalised harmonic mean of the coefficient of variation of the entries of $\bX$. 
Owing to the properties of the harmonic mean, we are able to infer the following corollary.

\begin{corollary}
\label{crr:prop_harm_mean}
    Let $\bX\sim\mu\in\PP_2(\erre^n)$ be a random vector. Denoted with $\bX^*$ its principal components, we have that
    \begin{enumerate}
        \item \label{enum1} $\min_{i=1,\dots,n}CV(X_i^*)\le G_2(\bX)\le \max_{i=1,\dots,n}CV(X_i^*)$. In particular,  if all the principal components of $\bX$ have the same one dimensional coefficient of variation, the multidimensional $G_2$ on the vector of the principal components  attains the same value.
        \item \label{enum2} If we replace $k$ principal components of $\bX^*$, which we denote with $\bK^*,$ with $k$ uncorrelated random variables whose coefficients of variation are all equal to the $G_2(\bK^*)$ we have that $G_2(\bX)=G_2(\bR^*,\bK^*)$, where $\bR^*$ is a random vector containing the $n-k$ principal components not contained in $\bK^*$.
    \end{enumerate}
\end{corollary}

\begin{proof}
    Statement \ref{enum1} follows from the additivity of the harmonic mean. Statement \ref{enum2} follows from the associativity of the harmonic mean.
\end{proof}

Notice that Corollary \ref{crr:prop_harm_mean} does not hold for the classic Voinov-Nikulin's coefficient of variation.

\subsubsection{Rising Tide Property}

We now consider the rising tide property.
Notice that the same argument we use to prove $G_2$ possesses this property can be adapted to show that $\gamma_{VN}$ possesses it as well.
First of all, we notice that if $\Sigma$ is the covariance matrix of a random vector $\bX$, then the vector $\bX+\bc$ has the same covariance matrix for any $\bc\in\erre^n$.

\begin{theorem}
    The $G_2$ multivariate coefficient of variation satisfies the rising tide property.
\end{theorem}

\begin{proof}
    Let $\bX$ be a random vector with covariance matrix $\Sigma$ and mean $\bm$.
    Given $\bc$ such that $\bc^T\Sigma^{-1}\bm\ge 0$, we have that
    \begin{align*}
        G_2(\bX+\bc)=\sqrt{\frac{n}{(\bm+\bc)^T\Sigma^{-1}(\bm+\bc)}}&=\sqrt{\frac{n}{\bm^T\Sigma^{-1}\bm+2\bc^T\Sigma^{-1}\bm+\bc^T\Sigma^{-1}\bc}}\\
        &\le\sqrt{\frac{n}{\bm^T\Sigma^{-1}\bm}}=G_2(\bX),
    \end{align*}
    which concludes the proof.    
\end{proof}

\subsubsection{Cloning}
The $G_2$ coefficient of variation possesses the cloning property. 
Indeed, given a random vector $\bX$ with mean $\bm$ and covariance matrix $\Sigma$, we have that the $2n$ dimensional random vector $(\bX,\bX)$ has mean $(\bm,\bm)$ and covariance matrix
\begin{equation}
\label{eq:double_cov}
    \Sigma_{(\bX,\bX)}=\begin{bmatrix}
        \Sigma, &0\\
        0, &\Sigma
    \end{bmatrix},
\end{equation}
where $0$ denotes the $n\times n$ null matrix. We then have that
\begin{equation}
    \Sigma_{(\bX,\bX)}^{-1}=\begin{bmatrix}
        \Sigma^{-1}, &0\\
        0, &\Sigma^{-1}
    \end{bmatrix},
\end{equation}
thus $(\bm,\bm)^T\Sigma_{(\bX,\bX)}^{-1}(\bm,\bm)=2\bm^T\Sigma^{-1}\bm$.
In particular we infer that
\[
G_2\big((\bX,\bX)\big)=\sqrt{\frac{2n}{2\bm^T\Sigma^{-1}\bm}}=\sqrt{\frac{n}{\bm^T\Sigma^{-1}\bm}}=G_2(\bX).
\]
Notice that, while $G_2$ possess the cloning property, $\gamma_{VN}$ does not due to the lack of the corrective term $\sqrt{n}$.

\subsubsection{Dimension Stability}
Lastly, we show that $G_2$ is dimension stable.
Indeed, let us consider a sequence of probability measures $\{\mu^{(n)}\}_{n\in\enne}$ as in Definition \ref{def:dimstable}.
Denoted with $\bX^{(n)}$ the random vector associated with $\mu^{(n)}=(m^{(n)}_1,\dots,m^{(n)}_n)$, with $\bm^{(n)}$ the mean of $\mu^{(n)}$, and with $\sigma^{(n)}_{n,n}$ the standard deviation of $X_n^{(n)}$, we have that
\begin{align}
\label{eq:hdns}
    G_2(\mu^{(n)})=\sqrt{n}\;\gamma_{VN}(\mu^{(n)})=\sqrt{\frac{n}{\sum_{i=1}^n\frac{(m^{(i)}_i)^2}{(\sigma^{(i)}_{i,i})^2}}}=\sqrt{\frac{1}{\frac{1}{n}\sum_{i=1}^n\frac{(m^{(i)}_i)^2}{(\sigma^{(i)}_{i,i})^2}}}.
\end{align}
The denominator on the right-hand side of equation \eqref{eq:hdns}  is a Cesaro's sum, which converges to $\lim_{i\to\infty}\frac{(m^{(i)}_i)^2}{(\sigma^{(i)}_{i,i})^2}$ as long as the latter limit exists.
Again, notice that, while $G_2$ is dimension stable, $\gamma_{VN}$ is not, as it lacks the corrective term $\sqrt{n}$ that makes the coefficient converge to a limit different from zero, when $n$ goes to $\infty$ in \eqref{eq:hdns}.

\begin{remark}[Influence Function of $G_2$]
    To conclude, we discuss the influence function of $G_2$, which measures its robustness when we add a perturbation concentrated in a singular point.
    Formally, the influence function is defined as
    \[
        IF_{G_2}(x;\mu):=\lim_{\epsilon\to0}\frac{G_2(\mu_\epsilon)-G_2(\mu)}{|\epsilon|}=\frac{\partial G_2(\mu_\epsilon)}{\partial  \epsilon}\Big|_{\epsilon=0},
    \]
    where $\mu_\epsilon=(1-\epsilon)\mu+\epsilon\delta_x$ and $\delta_x$ is the Dirac's delta centred in $x$.
    Notice that, since $G_2$ is equal to $\sqrt{n}\gamma_{VN}$ its influence function is $\sqrt{n}$ times the influence function of $\gamma_{VN}$, thus
    \[
        IF_{G_2}(x;\mu):=\frac{\sqrt{n}}{2\big(\bm^T\Sigma^{-1}\bm\big)^{\frac{3}{2}}}\Big((\bm^T\Sigma^{-1}\bm)^2-2\bm^T\Sigma^{-1}(\bx-\bm)\Big)       
    \]
    as shown in \cite{aerts2015multivariate}.
    Since all influence functions of the other MCV have been computed in \cite{aerts2015multivariate},  we will omit their study.   
\end{remark}


%

\subsection{A \texorpdfstring{$G_q$}{Gp} Multivariate Coefficient of Variation}
We can generalise the squared multidimensional  coefficient in \eqref{eq:G_2_main} considering the $L_q$ norm rather than the $L_2$ norm in \eqref{eq:G_2_main}. 
For any given $q\ge 1$ and $\mu \in \PP_q(\R^n)$, we define
\begin{equation}
\label{def:Ginipdef}
    G_q(\bX)=\Bigg(\frac{1}{2||W_\mu^{ZCA}\bm||^q_q}\int_{\erre^n}\int_{\erre^n}||W^{ZCA}_{\mu}(\bx-\by)||^q_q\mu(d\bx)\mu(d\by)\Bigg)^{\frac{1}{q}},
\end{equation}
where $W_\mu^{ZCA}$ is the ZCA-cor whitening matrix associated with $\bX\sim\mu$.  
Notice that, when $q=2$, we obtain $G_2$ since $||W_\mu^{ZCA}\bm||_2^2=(W_\mu^{ZCA}\bm)^TW_\mu^{ZCA}\bm=\bm^T\Sigma^{-1}\bm$.
First, we notice that $G_2$ is the only MCV that is coherent.
Indeed, given $q\neq 2$ and $X\sim\mu\in \PP_q(\erre)$, we have that $W_\mu^{ZCA}=\frac{1}{\sigma}$ thus
\begin{equation}
    G_q(X)=\frac{\sigma}{|m|}\Bigg(\frac{1}{2\sigma^q}\int_{\erre^n}\int_{\erre^n}|x-y|^q\mu(dx)\mu(dy)\Bigg)^{\frac{1}{q}}=\Bigg(\frac{1}{2|m|^q}\int_{\erre^n}\int_{\erre^n}|x-y|^q\mu(dx)\mu(dy)\Bigg)^{\frac{1}{q}}.
\end{equation}
In particular, we have that $G_q(X)=CV(X)$ if and only if
\[
    \int_{\erre^n}\int_{\erre^n}|x-y|^q\mu(d\bx)\mu(d\by)=2\sigma^q,
\]
which is not true in general.
By the same argument, we have that $G_q$ is not dimension stable.
On univariate random variables, $G_q$ does not depend on $\sigma$, whereas it is easy to see that for a multidimensional random vector $\bX$ with uncorrelated entries $G_q(\bX)$ do depend on the variances of each entry $\sigma_i$.
We then conclude that $G_q$ is not SUF whenever $q\neq 2$.
Moreover, following the argument used in \cite{auricchio2024extending}, it is possible to show that each $G_q$ for $q\neq 2$ is not scale invariant.
However, owing to the fact that the ZCA-cor whitening process is scale stable, $G_q$ is invariant under change of unit of measurement, i.e. $G_q(\bX)=G_q(Q\bX)$ for every diagonal matrix $Q$ whose diagonal values are positive.
Finally, we notice that every $G_q$ satisfies the cloning property.
To conclude we study the limit of $G_q$ for $q\to\infty$.

\begin{theorem}
    Given $\bX\sim\mu\in \PP_q(\erre^n)$ with bounded support, we have that
    \[
\lim_{q\to\infty}G_q(\bX)=\frac{\max_{i=1,\dots,n}range(X_i^*)}{||W_\mu^{ZCA}\bm||_{\infty}},
    \]
    where $range(X_i^*)$ is the diameter of the support of the $i$-th principal component and $||W_\mu^{ZCA}\bm||_{\infty}=\max_{i=1,\dots,n}|(W_\mu^{ZCA}\bm)_i|$ is the $l_\infty$-norm of the whitened mean vector $\bm^*$.
\end{theorem}

\begin{proof}
    Let $\bX\sim\mu\in \PP_q(\erre^n)$ be a random vector with bounded support.
    By definition, we have that
    \begin{align*}
        G_q(\bX)=\frac{1}{2^{\frac{1}{q}}||W_\mu^{ZCA}\bm||_q}\Bigg( \int_{\erre^{n}}\int_{\erre^n}\sum_{i=1}^n|(W_\mu^{ZCA}(\bx-\by))_i|^q\mu(d\bx)\mu(d\by) \Bigg)^{\frac{1}{q}}.
    \end{align*}
    First, we notice that $\lim_{q\to\infty}2^{\frac{1}{q}}||W_\mu^{ZCA}\bm||_q=||W_\mu^{ZCA}\bm||_{\infty}$, since $2^{\frac{1}{q}}\to 1$ and, it is well-known that the $l_\infty$-norm is the limit of the $l_q$-norm for $q\to\infty$.
    To conclude the proof, we then need to show that
    \[
        \lim_{q\to\infty}\Bigg(\int_{\erre^n}\int_{\erre^n}\sum_{i=1}^n|W_\mu^{ZCA}((\bx-\by))_i|^q\mu(d\bx)\mu(d\by)\Bigg)^{\frac{1}{q}}=\max_{i=1,\dots,n}range(X_i^*),
    \]
    where $X_i^*$ is the $i$-th principal component of $\bX$.
    To simplify the notation, from now on, we adopt the following notation
    \[
    ||(W_\mu^{ZCA}(\bx-\by))_i||^q_q=\int_{\erre^n}\int_{\erre^n}|(W_\mu^{ZCA}(\bx-\by))_i|^q\mu(d\bx)\mu(d\by),
    \]
    where $i=1,\dots,n$.
    We then have
    \begin{align*}
        \Big(\sum_{i=1}^n||(W_\mu^{ZCA}(\bx-\by))_i||^q_q&\Big)^{\frac{1}{q}}-\max_{i=1,\dots,n}range(X_i^*)\\
        &=\Big(\sum_{i=1}^n||(W_\mu^{ZCA}(\bx-\by))_i||^q_q\Big)^{\frac{1}{q}}- \max_{i=1,\dots,n} ||(W_\mu^{ZCA}(\bx-\by))_i||_q\\
        &\quad+\max_{i=1,\dots,n} \max_{i=1,\dots,n} ||(W_\mu^{ZCA}(\bx-\by))_i||_q  -\max_{i=1,\dots,n}range(X_i^*).
    \end{align*}
    Owing to the fact that the $L_q$-norm of a function with respect to the measure $\mu\otimes\mu$ converges to the $L_\infty$-norm of the function with respect to $\mu\otimes\mu$, we have that
    \[
    \lim_{q\to\infty}\Bigg(\int_{\erre^n}\int_{\erre^n}|W_\mu^{ZCA}((\bx-\by))_i|^q\mu(d\bx)\mu(d\by)\Bigg)^{\frac{1}{q}}=\max_{x,y}|W_\mu^{ZCA}((\bx-\by))_i|=range(X_i^*).
    \]
    Since the $\max$ operator is continuous, we infer that
    \[
    \lim_{q\to\infty}\max_{i=1,\dots,n} \Bigg(\int_{\erre^n}\int_{\erre^n}|W_\mu^{ZCA}((\bx-\by))_i|^q\mu(d\bx)\mu(d\by)\Bigg)^{\frac{1}{q}}  -\max_{i=1,\dots,n}range(X_i^*)=0.
    \]
    We then need to show that
    \begin{align}
        \lim_{q\to\infty}&\Big(\sum_{i=1}^n||(W_\mu^{ZCA}(\bx-\by))_i||^q_q\Big)^{\frac{1}{q}}=\lim_{q\to\infty}\max_{i=1,\dots,n} ||(W_\mu^{ZCA}(\bx-\by))_i||_q.
    \end{align}
    Notice that
    \[
        \max_{i=1,\dots,n}||(W_\mu^{ZCA}(\bx-\by))_i||_q\le \Big(\sum_{i=1}^n||(W_\mu^{ZCA}(\bx-\by))_i||^q_q\Big)^{\frac{1}{q}}\le n^{\frac{1}{q}}\max_{i=1,\dots,n}||(W_\mu^{ZCA}(\bx-\by))_i||_q,
    \]
    which allows us to conclude the proof since $n$ is a fixed parameter.
\end{proof}

\begin{remark}
Note that $G_\infty$ defines a natural generalisation of the normalised univariate range to a multivariate setting.

\end{remark}

\section{Other Multivariate Coefficients of Variation}
\label{sec:other}

In this Section, we study the properties of the other MCVs introduced in Section \ref{sec:otherCVs}.
We do not consider the Voinov-Nikulin's coefficient of variation $\gamma_{VN}$, as it has been already discussed in the previous section along with $G_2$.
As shown in \cite{aerts2015multivariate}, only $\gamma_{VN}$ is scale invariant, so we omit a further analysis of this property.

\subsection{The Reyment's Coefficient of Variation} 
We start from the Reyment's coefficient of variation, that is
$$\gamma_{R}(\mu)=\sqrt{\frac{det(\Sigma)^{\frac{1}{n}}}{\bm^T\bm}}.$$
Note  that this coefficient is coherent, since $\det(\Sigma)=\sigma^2$ when $n=1$.
Despite being coherent, $\gamma_R$ does not possess any other of the properties outlined in Section \ref{sec:properties}.
First we show that $\gamma_R$ does not possess the cloning property.
Indeed, let $\bX$ be a random vector with mean $\bm$ and covariance $\Sigma$, we have that an independent coupling $(\bX,\bX)$ has mean $(\bm,\bm)$ and covariance matrix as in \eqref{eq:double_cov}.
We then have that $det(\Sigma_{(\bX,\bX)})=det(\Sigma)^2$ and $(\bm,\bm)^T(\bm,\bm)=2\bm^T\bm$, thus
\[
\gamma_R((\bX,\bX))=\sqrt{\frac{det(\Sigma_{(\bX,\bX)})^{\frac{1}{2n}}}{\bm^T\bm}}=\sqrt{\frac{det(\Sigma)^{\frac{1}{n}}}{2\bm^T\bm}}<\sqrt{\frac{det(\Sigma)^{\frac{1}{n}}}{\bm^T\bm}}=\gamma_R(\bX).
\]
Therefore, we have that $\gamma_R$ does not satisfy the cloning property.
We then show that $\gamma_R$ does not satisfy the rising tide property.
Let us consider a random vector $\bX=(X_1,X_2)$ whose mean is $\bm=(3,3)$ and covariance matrix 
\begin{equation}
\label{eq:sigmarising}
    \Sigma=\begin{bmatrix}
    1, &1\\
    1, &2
\end{bmatrix}.
\end{equation}
Let us now consider the vector $\bc=(1,-2)$ and $\bY=\bX+\bc$.
A simple computation shows that
\begin{equation*}
    \Sigma^{-1}=\begin{bmatrix}
    2, &-1\\
    -1, &1
\end{bmatrix},
\end{equation*}
hence $\bc^T\Sigma^{-1}\bm=3>0$.
To conclude, we need to show that $\gamma_R(\bY)>\gamma_R(\bX)$.
The mean of $\bY$ is $(4,1)$ while the covariance matrix of $\bY$ remains $\Sigma$ defined as in \eqref{eq:sigmarising}.
We have that $\gamma_R(\bX)=\sqrt{\frac{1}{18}}<\sqrt{\frac{1}{17}}=\gamma_R(\bY)$, thus $\gamma_R$ does not satisfy the rising tide property.
Moreover, we have that $\gamma_R$ does not possess the Splitting Uncorrelated Features property.
Indeed, let us assume that $\gamma_R(\bX)=G(CV(X_1),\dots,CV(X_n))$, where $G$ is a suitable function and $\bX=(X_1,\dots,X_n)$ is a vector whose components are uncorrelated.
For the sake of simplicity, let us assume that $n=2$, so that $\bX=(X_1,X_2)$, $\bm=(1,1)$, and $\Sigma=Id_2$, where $Id_2$ is the $2\times 2$ identity matrix.
We then have that $CV(X_1)=CV(X_2)=1$ thus
\[
    G(1,1)=\gamma_R(\bX)=\frac{1}{\sqrt{2}}.
\]
Let us now consider $\bY=(Y_1,Y_2)=(2 X_1,X_2)$.
The mean of $\bY$ is $(2,1)$ and has covariance matrix
\[
\Sigma=\begin{bmatrix}
    4, &0\\
    0, &1
\end{bmatrix}.
\]
However, we have that $CV(Y_1)=CV(Y_2)=1$, thus we have that
\[
    \frac{1}{\sqrt{2}}=\gamma_R(\bX)=G(1,1)=\gamma_R(\bY)=\sqrt{\frac{2}{4+1}}=\sqrt{\frac{2}{5}},
\]
which is a contradiction.
Finally, we show that $\gamma_R$ is not dimension stable.
Consider a sequence of $n$ dimensional random vectors $\bX^{(n)}\sim\mu^{(n)}$ as in Definition \ref{def:dimstable}.
Furthermore assume that every marginal of the random vector has mean $m_i=m\neq0$, so that $\bm^{(n)}=\bm=(m,\dots,m)\in\erre^n$ and variance $(\sigma^{(n)}_i)^2=\sigma^2$, so that $\Sigma=\sigma^2 Id_n$.
It is then easy to see that $det(\Sigma_n)=\sigma^{2n}$ and that $\bm^T\bm=nm^2$, moreover $\lim_{n\to\infty}CV(\mu^{(n)}_n)=\frac{\sigma}{|m|}$.
However, we have that
\[
\lim_{n\to\infty}\gamma_R(\mu^{(n)})=\lim_{n\to\infty}\sqrt{\frac{\sigma^2}{nm^2}}=0,
\]
proving that $\gamma_R$ is not dimension stable.

\subsection{The Van Valen's Coefficient of Variation}

We consider the Van Valen's coefficient of variation, that is
$$\gamma_{VV}(\mu)=\sqrt{\frac{tr(\Sigma)}{\bm^T\bm}}.$$
Note that this coefficient is coherent since $tr(\Sigma)=\sigma^2$ when $n=1$.
Moreover, $\gamma_{VV}$ possesses the cloning property.
Indeed, given a random vector $\bX$ with mean $\bm$ and covariance matrix $\Sigma$, we have that the coupling $(\bX,\bX)$ has mean $(\bm,\bm)$ and covariance matrix as in \eqref{eq:double_cov}.
We then have that $(\bm,\bm)^T(\bm,\bm)=2\bm^T\bm$ and $tr(\Sigma_{(\bX,\bX)})=2tr(\Sigma)$, thus
\[
    \gamma_{VV}\big((\bX,\bX)\big)=\sqrt{\frac{tr(\Sigma_{(\bX,\bX)})}{(\bm,\bm)^T(\bm,\bm)}}=\sqrt{\frac{2tr(\Sigma)}{2\bm^T\bm}}=\gamma_{VV}(\bX).
\]
The same example used for $\gamma_R$ can be employed to show that also $\gamma_{VV}$ does not possess the rising tide property.
%
%
Moreover, $\gamma_{VV}$ is not SUF.
Toward a contradiction, let $\bX=(X_1,X_2)$ be a random vector with mean $(2,1)$ and covariance matrix $\Sigma=Id_2$.
We then have a contradiction since 
\[
\sqrt{\frac{2}{5}}=\gamma_{VV}(\bX)=G(2,1)=\gamma\big((2X_1,X_2)\big)=\sqrt{\frac{5}{17}},
\]
since the mean of $(2X_1,X_2)$ is $(4,1)$ and its covariance matrix $\begin{bmatrix}
    4, &0\\ 0, &1
\end{bmatrix}$.
%

%
Lastly, we show that $\gamma_{VV}$ is dimension stable.

\begin{theorem}
    The Van Valen's coefficient of variation is dimension stable.
    In particular, given a sequence of probability measures $\{\mu^{(n)}\}_{n\in\enne}$ as in Definition \ref{def:dimstable} and denoted with $L$ the limit coefficient of variation of the marginals of the sequence, that is $\lim_{n\to\infty}\frac{\sigma_{n,n}^{(n)}}{|m_n^{(n)}|}=L$, we have that
    \begin{equation}
        \lim_{n\to\infty}\gamma_{VV}(\mu^{(n)})=\sqrt{L}.
    \end{equation}
\end{theorem}

\begin{proof}
    By hypothesis, we have that for every $\epsilon>0$ there exists $N$ such that
    \[
    (L-\epsilon)(m^{(n)}_n)^2\le(\sigma^{(n)}_{n,n})^2\le(L+\epsilon)(m^{(n)}_n)^2
    \]
    for every $n\ge N$.
    Thus, for $n\ge N$, we infer that
    \begin{align}
        \gamma_{VV}(\mu^{(n)}) &= \sqrt{\frac{\sum_{i=1}^n(\sigma^{(n)}_{i,i})^2}{\sum_{i=1}^n (m_i^{(n)})^2}} = \sqrt{\frac{\frac{1}{n}\sum_{i=1}^n(\sigma^{(n)}_{i,i})^2}{\frac{1}{n}\sum_{i=1}^n (m_i^{(n)})^2}} = \sqrt{\frac{\frac{1}{n}(\sum_{i=1}^N(\sigma^{(n)}_{i,i})^2+\sum_{i=N}^n(\sigma^{(n)}_{i,i})^2)}{\frac{1}{n}(\sum_{i=1}^N (m_i^{(n)})^2+\sum_{i=N}^n(m_i^{(n)})^2)}}
    \end{align}
    hence
    \begin{align*}
        \sqrt{\frac{\frac{1}{n}(\sum_{i=1}^N(\sigma^{(n)}_{i,i})^2+\sum_{i=N}^n(L-\epsilon)(m_i^{(n)})^2)}{\frac{1}{n}(\sum_{i=1}^N (m_i^{(n)})^2+\sum_{i=N}^n(m_i^{(n)})^2)}}&\le \gamma_{VV}(\mu^{(n)})\\
        &\le \sqrt{\frac{\frac{1}{n}(\sum_{i=1}^N(\sigma^{(n)}_{i,i})^2+\sum_{i=N}^n(L+\epsilon)(m_i^{(n)})^2)}{\frac{1}{n}(\sum_{i=1}^N (m_i^{(n)})^2+\sum_{i=N}^n(m_i^{(n)})^2)}}.
    \end{align*}
    By taking the limit for $n\to\infty$, we conclude the proof.
\end{proof}

\subsection{The Albert and Zhang's Coefficient of Variation} 

First, we notice that the Albert and Zhang's coefficient of variation $\gamma_{AZ}$ is coherent as we have that $\bm^T\Sigma\bm=m^2\sigma^2$ when $n=1$, hence
\[
    \gamma_{AZ}(\mu)=\sqrt{\frac{m^2\sigma^2}{m^4}}=\frac{\sigma}{|m|}
\]
when $\mu$ is the probability distribution supported over $\erre$.
Despite being coherent, $\gamma_{AZ}$ does not possess the rising tide property nor satisfy the cloning property, the SUF property, nor is dimension stable.
Let us start from the rising tide property.
Let $\bX=(X_1,X_2)$ be a random vector such that $\bm=(1,0.1)$ and 
\begin{equation}
\label{eq:sigmaAZRT}
    \Sigma=\begin{bmatrix}
        1, &0\\
        0, &100
    \end{bmatrix}.
\end{equation} 
It is then easy to see that $\gamma_{AZ}(\bX)=\frac{2}{(1.001)^2}<\sqrt{2}$.
Let us then consider $\bc=(0,0.99)$.
Since $\Sigma$ is a diagonal matrix, we have that $\bc^T\Sigma^{-1}\bm>0$. Moreover, the mean of $\bX+\bc$ is $(1,1)$ and its covariance matrix remains $\Sigma$ as in \eqref{eq:sigmaAZRT}.
We then have
\[
    \gamma_{AZ}(\bX+\bc)=\sqrt{\frac{(1,1)^T\Sigma(1,1)}{((1,1)^T(1,1))^2}}=\frac{101}{4}>2>\gamma_{AZ}(\bX),
\]
which shows  that $\gamma_{AZ}$ does not possess the rising tide property.
We now move to the cloning property.
Given $\bX$ a random vector with mean $\bm$ and covariance $\Sigma$, then we have that an independent coupling $(\bX,\bX)$ has mean $(\bm,\bm)$ and covariance matrix as in \eqref{eq:double_cov}.
We then have that $(\bm,\bm)^T\Sigma_{(\bX,\bX)}(\bm,\bm)=2\bm^T\Sigma\bm$, hence
\[
\gamma_{AZ}\big((\bX,\bX)\big)=\sqrt{\frac{2\bm^T\Sigma\bm}{(2\bm^T\bm)^2}}=\frac{1}{\sqrt{2}}\gamma_{AZ}(\bX).
\]
%

%
%

%
To show that $\gamma_{AZ}$ does not possess the SUF property, it suffices to mimic the argument used for $\gamma_R$ and $\gamma_{VV}$.
Indeed, consider a Gaussian vector $\bX=(X_1,X_2)$ with mean $\bm=(1,1)$ and covariance matrix $\Sigma=Id_2$.
Toward assume that there exists $G$ such that $\gamma_{AZ}(\bX)=G(CV(X_1),CV(X_2))$, then we have
\[
    \sqrt{\frac{1}{2}}=\gamma_{AZ}(\bX)=G(1,1)=\gamma_{AZ}((2X_1,X_2))=\sqrt{\frac{17}{25}},
\]
which allows us to conclude that $\gamma_{AZ}$ does not possess the SUF property.
To conclude, we notice that $\gamma_{AZ}$ is not dimension stable.
As per the $\gamma_R$, consider a sequence of independent Gaussian distributions $\mu^{(n)}$ that satisfies the condition outlined in Definition \ref{def:dimstable} and such that every marginal $\mu^{(n)}_i$ has mean $m$ and variance $\sigma^2$.
By definition, we have that
\[
\gamma_{AZ}(\mu^{(n)})=\sqrt{\frac{nm^2\sigma^2}{(nm^2)^2}}=\frac{\sigma}{\sqrt{n}|m|}=0
\]
for every $n$, we thus conclude $\lim_{n\to\infty}\gamma_{AZ}(\mu^{(n)})=0$.

\begin{remark}[Correcting $\gamma_R$ and $\gamma_{AZ}$]
    As we have seen, both $\gamma_R$ and $\gamma_{AZ}$ are not dimension stable.
    However, as for $\gamma_{VN}$, this property can be recovered by adding a corrective term $\sqrt{n}$ to the two MCVs.
    Let us consider $\gamma_{AZ}$.
    Given a sequence $\{\mu^{(n)}\}_{n\in\enne}$ as in Definition \ref{def:dimstable}, we have
    \[
    \sqrt{n}\gamma_{AZ}(\mu^{(n)})=\sqrt{\frac{n\bm^T\Sigma\bm}{(\bm^T\bm)^2}}=\sqrt{\frac{\frac{1}{n}\sum_{i=1}^n m_i^2\sigma^2_i}{(\frac{1}{n}\sum_{i=1}^n m_i^2)^2}}.
    \]
    Owing again to the properties of Cesaro's sums, it is easy to see that 
    \[
    \lim_{n\to\infty}\sqrt{n}\gamma_{AZ}(\mu^{(n)})=\lim_{n\to\infty}CV(X_n^{(n)}).
    \]
    Thus $\sqrt{n}\gamma_{AZ}$ is dimension stable.
    Let us now consider $\gamma_R$ and $\{\mu^{(n)}\}_{n\in\enne}$ a sequence as in Definition \ref{def:dimstable}, we then have
    \[
    \sqrt{n}\gamma_{R}(\mu^{(n)})=\sqrt{\frac{n(det(\Sigma))^{\frac{1}{n}}}{\bm^T\bm}}=\sqrt{\frac{ (\prod_{i=1}^n\sigma_i^2)^{\frac{1}{n}}}{\frac{1}{n}\sum_{i=1}^n m_i^2}}.
    \]
    Let us consider $(\prod_{i=1}^n\sigma_i^2)^{\frac{1}{n}}$, by taking the logarithm, we have that 
    \[
    ln\big((\prod_{i=1}^n\sigma_i^2)^{\frac{1}{n}}\big)=\frac{1}{n}\sum_{i=1}^n ln(\sigma^2_i).
    \]
    Thus, owing again to the properties of Cesaro's sums, we have that
    \[
    \lim_{i\to\infty}\sqrt{n}\gamma_{R}(\mu^{(n)})=\lim_{i\to\infty}CV(X_i^{(n)}).
    \]
    Thus $\sqrt{n}\gamma_{R}$ is dimension stable.
\end{remark}

\subsection{The T Coefficient of Variation} 
We now consider the multivariate measure of inequality recently introduced in \cite{toscani2024}:
\[
\mathcal{G}(\mu)=\sqrt{\frac{1}{2\bm^T\Sigma^{-1}\bm}}\int_{\erre^n}\int_{\erre^n}\sqrt{(\bx-\by)^T\Sigma^{-1}(\bx-\by)}\mu(d\bx)\mu(d\by).
\]

It is easy to see that it is scale invariant, satisfies the rising tide property, and the cloning property.
Notice however that $\mathcal{G}$ is not coherent.
Indeed, given a probability measure $\mu\in\PP(\erre)$ with mean $m$ and variance $\sigma^2$, we have that
\[
    \mathcal{G}(\mu)=\frac{\sigma}{\sqrt{2}|m|}\int_{\erre}\int_{\erre}\frac{|x-y|}{\sigma}\mu(dx)\mu(dy)=CV(\mu)\int_{\erre}\int_{\erre}\frac{|x-y|}{\sqrt{2}\sigma}\mu(dx)\mu(dy)\neq CV(\mu).
\]

To conclude we show that, albeit $\mathcal{G}(\mu)\neq\sqrt{n}\gamma_{VN}(\mu)$, we have that asymptotically $\mathcal{G}$ behaves as $G_2$ on any sequence of probability measures $\{\mu^{(n)}\}_{n\in\enne}$ as in Definition \ref{def:dimstable}.

\begin{theorem}
    There exists a constant $c$ such that
    \[
    c\sqrt{n}\;\gamma_{VN}(\mu^{(n)})\le \mathcal{G}(\mu^{(n)}) \le \sqrt{n}\;\gamma_{VN}(\mu^{(n)}),
    \]
    for any given a sequence of probability measures $\{\mu^{(n)}\}_{n\in\enne}$ as in Definition \ref{def:dimstable} such that $\mu^{(n)}\in\PP_3(\erre^n)$ for every $n\in\enne$.
     In particular, $\mathcal{G}(\mu) \sim O(\sqrt{n})\gamma_{VN}(\mu)$.
\end{theorem}

\begin{proof}
    Let $\mu\in\PP_3(\erre)$, by Jensen's inequality, we have that
    \begin{align*}
        \mathcal{G}(\mu)&=\sqrt{\frac{1}{2\bm^T\Sigma^{-1}\bm}}\int_{\erre^n}\sqrt{(\bx-\by)^T\Sigma^{-1}(\bx-\by)}\mu(d\bx)\mu(d\by)\\
        &=\sqrt{\frac{1}{2\bm^T\Sigma^{-1}\bm}}\Bigg(\bigg(\int_{\erre^n}\int_{\erre^n}\sqrt{(\bx-\by)^T\Sigma^{-1}(\bx-\by)}\mu(d\bx)\mu(d\by)\bigg)^2\Bigg)^{\frac{1}{2}}\\
        &\le \sqrt{\frac{1}{2\bm^T\Sigma^{-1}\bm}}\Bigg(\int_{\erre^n}\int_{\erre^n}(\bx-\by)^T\Sigma^{-1}(\bx-\by)\mu(d\bx)\mu(d\by)\Bigg)^{\frac{1}{2}}\\
        &=G_2(\mu)=\sqrt{n}\;\gamma_{VN}(\mu),
    \end{align*}
    which allows us to conclude one half of the proof.
    Let us now consider the other inequality and let $W_\mu^{ZCA}$ be the ZCA-cor whitening matrix associated to $\mu$.
    Then, by the change of variable $\bx^*=W_\mu^{ZCA}\bx$ and $\by^*=W_\mu^{ZCA}\by$, we have that
    \begin{align*}
        \int_{\erre^n}\int_{\erre^n}\sqrt{(\bx-\by)^T\Sigma^{-1}(\bx-\by)}\mu(d\bx)\mu(d\by)=\int_{\erre^n}\int_{\erre^n}\sqrt{(\bx^*-\by^*)^T(\bx^*-\by^*)}\mu^*(d\bx^*)\mu^*(d\by^*),
    \end{align*}
    where $\mu^*$ is the whitened probability measure associated with $\mu$.
    We then have that
    \begin{align*}
        \int_{\erre^n}\int_{\erre^n}\sqrt{(\bx^*-\by^*)^T(\bx^*-\by^*)}\mu^*(d\bx^*)&\mu^*(d\by^*)\\
        &\ge\frac{1}{\sqrt{n}} \int_{\erre^n}\int_{\erre^n}\sum_{i=1}^{n}|x^*_i-y^*_i|\mu^*(d\bx^*)\mu^*(d\by^*)\\
        &=\frac{1}{\sqrt{n}}\sum_{i=1}^{n} \int_{\erre^n}\int_{\erre^n}|x^*_i-y^*_i|\mu^*(d\bx^*)\mu^*(d\by^*)\\
        &\ge \frac{1}{\sqrt{n}}\sum_{i=1}^{n} \int_{\erre^n}|x^*_i-m^*_i|\mu^*(d\bx^*)
    \end{align*}
    where the last inequality comes from Jensen's inequality, since
    \[
        \int_{\erre^n}|x_i^*-y_i^*|\mu^*(d\by^*)\ge \bigg|x^*_i-\int_{\erre^n}y^*_i\mu^*(d\by^*)\bigg|=\big|x^*_i-m^*_i\big|.
    \]
    To conclude, it suffice to show that
    \[
    \int_{\erre^n}|x_i^*-m_i^*|\mu^*(d\bx^*)\ge c,
    \]
    for every $i=1,\dots,n$.
    Indeed, for every $R>0$, we have that
    \begin{align*}
        1&=\int_{\erre^n}|x_i^*-m_i^*|^2\mu^*(d\bx^*)\\
        &=\int_{\{|x^*_i-m^*_i|\ge R\}}|x^*_i-m^*_i|^2\mu^*(d\bx^*) + \int_{\{|x^*_i-m^*_i| < R\}}|x^*_i-m^*_i|^2\mu^*(d\bx^*)\\
        &\le \frac{1}{R}\int_{\{|x^*_i-m^*_i|\ge R\}}|x^*_i-m^*_i|^3\mu^*(d\bx^*)+ R\int_{\{|x^*_i-m^*_i| < R\}}|x^*_i-m^*_i|\mu^*(d\bx^*)\\
        &\le \frac{1}{R}\int_{\erre^n}|x^*_i-m^*_i|^3\mu^*(d\bx^*)+ R\int_{\erre^n}|x^*_i-m^*_i|\mu^*(d\bx^*)\\
        &\le \frac{1}{R}K_3+ RM_1\\
    \end{align*}
    where $M_1=\int_{\erre^n}|x^*_i-m^*_i|\mu^*(d\bx^*)$ and $K_3=\int_{\erre^n}|x^*_i-m^*_i|^3\mu^*(d\bx^*)$.
    Finally, we search for the value of $\bar R$ that minimizes $\frac{1}{R}K_3+ RM_1$.
    A simple computation shows that $\bar R = \sqrt{\frac{K_3}{M_1}}$, hence we have $1\le 2\sqrt{K_3 M_1}$, therefore
    \begin{equation}
        \int_{\erre^n}|x_i^*-m^*_i|\mu^*(d\bx^*)\ge c:=\frac{1}{4K_3}.
    \end{equation}
    We then conclude that
    \begin{align*}
        \int_{\erre^n}\int_{\erre^n}\sqrt{(\bx-\by)^T\Sigma^{-1}(\bx-\by)}\mu(d\bx)\mu(d\by)&\ge \frac{1}{\sqrt{n}}\sum_{i=1}^{n} \int_{\erre^n}|x^*_i-m^*_i|\mu^*(d\bx^*)\\
        &\ge \frac{1}{\sqrt{n}}\sum_{i=1}^{n} \frac{1}{4K_3}=\frac{\sqrt{n}}{4K_3},
    \end{align*}
    which concludes the proof.    
\end{proof}

To conclude, note that, by assuming $\mu$ regular enough, the same argument can be extended to any $G_q$ index to show that $G_q(\mu^{(n)})\sim O(\sqrt{n})\gamma_{VN}(\mu^{(n)})$ for every sequence of probability measures as in Definition \ref{def:dimstable}.

\section{Simulation studies}
\label{sec:6}

We complement our theoretical results with two numerical experiments. 
In the first experiment, we compute the multivariate coefficients of variation for a sequence of random samples simulated from multivariate Gaussian vectors of increasing dimension. 
In the second experiment, we consider a sequence of random vectors that describes the trajectory of $100$ points which move according to a Galton-like time series and compute the multivariate coefficients of variation at the the  end points.

\subsection{ Multivariate Gaussian Distribution}

In this first experiment, we consider an equally spaced sequence of dimensions $n=10,15,\dots,50$ and, for each of them, we sample $500$ points from a $n$-dimensional Gaussian vector, with covariance matrix $\Sigma=2Id_n$, where $Id_n$ is the $n\times n$ identity matrix and with mean either: \begin{enumerate*}[label=(\roman*)]
    \item $\bm=(2,\dots,2)\in\erre^n$; or
    \item $\bm$ an $n$-dimensional vector whose entries are sampled uniformly in $[1,2]$.
\end{enumerate*}

From the theoretical results presented in Section \ref{sec:ourCV} and \ref{sec:other}, we expect the five coefficients of variation to converge at different limits. 
Figures \ref{fig:Gaussian1} and \ref{fig:Gaussian2}  presents the results for the simulated data. 

\begin{figure}[t!]
    \begin{minipage}{0.475\textwidth}
        \centering
        \includegraphics[width=\textwidth]{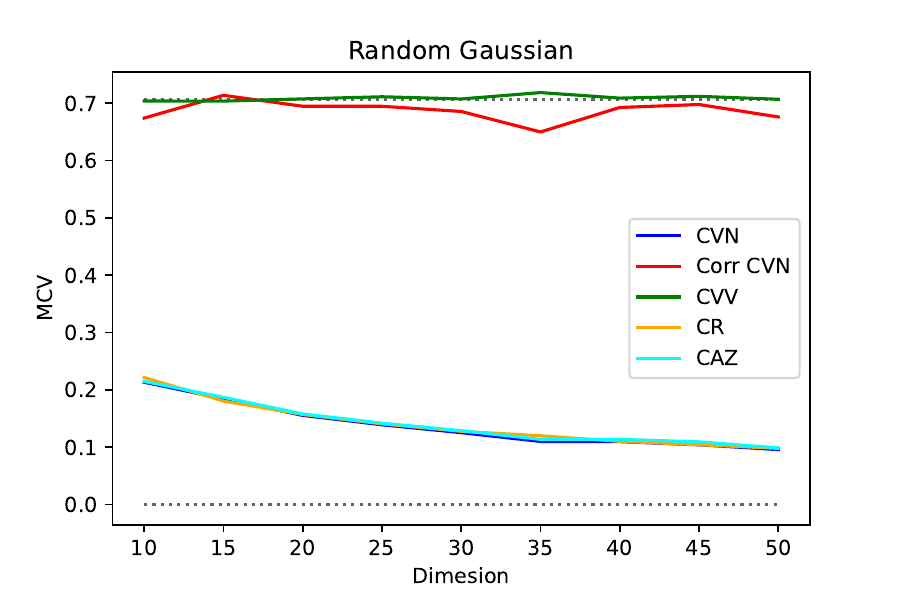}
        \caption{ Multivariate Coefficients of Variation for the $n$-dimensional Gaussian samples with means $\bm=(2,\dots,2)$.}
        \label{fig:Gaussian1}
    \end{minipage}\hfill
    \begin{minipage}{0.475\textwidth}
        \centering
        \includegraphics[width=\textwidth]{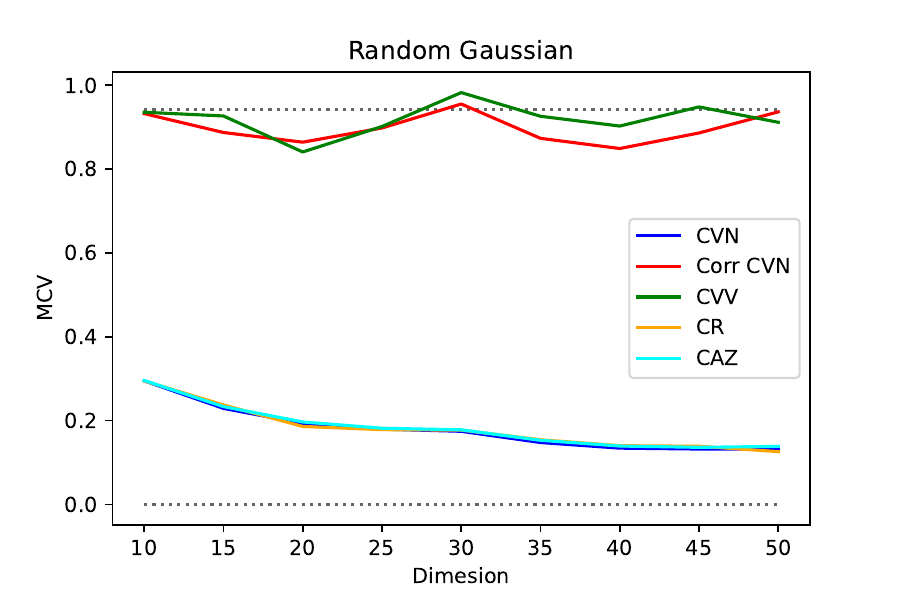}
        \caption{Multivariate Coefficients of Variation for the $n$-dimensional Gaussian samples with means sampled from a uniform distribution.}
        \label{fig:Gaussian2}
    \end{minipage}
\end{figure}

Figures \ref{fig:Gaussian1} and \ref{fig:Gaussian2} show that, in  line with the theoretical results: 
\begin{itemize}
    \item  $\gamma_{VN}$, $\gamma_{AZ}$, and $\gamma_{R}$ converge to zero at a rate equal to $\frac{1}{\sqrt{n}}$, as the dimension $n$ increases;  regardless of the mean values;
    \item $\gamma_{VV}$ and $G_2 = \gamma_{CorrVN}$ converge to $\sqrt{\frac{\sigma^2}{m^2}}=\frac{1}{\sqrt{2}}\sim 0.71$ when $\bm=(2,\dots,2)$ and to $\frac{\sqrt{2}}{\frac{3}{2}}=\frac{2\sqrt{2}}{3}\sim0.94$ when the means are sampled from a uniform distribution. 
    \item $\gamma_{VN}$, $\gamma_{AZ}$, and $\gamma_{R}$ converge to a value that is independent from the means of the distribution; whereas the limiting value of
  $\gamma_{VV}$ and $\gamma_{CorrVN}$ depend, more correctly, on the means, with a higher value in the case of a variable mean, whose sequence fluctuates more.  
\end{itemize} 

\subsection{Galton-like Trajectories}

In the second experiment, we consider a set of  particles whose position changes at every time step, under the effect of a random variable.
More precisely, if the position of a particle at time step $t$ is $x_t$, at the next time step the particle will move one unit to the right with probability $0.5$ or one unit to the left with probability $0.5$.
Given a time horizon $T$, each initial position of a particle induces a random trajectory of $T$ points, which we consider as a $T$-dimensional vector.
We denote such  trajectories as \textit{Galton-like trajectories}, as they are inspired from the experiments of Sir Francis Galton.
We consider $100$ particles, with different initial points and, correspondingly, $100$   $T$-dimensional random trajectories, where $T$ is the time-horizon we consider: the number of times we update the position of each particle.
In our setting, we consider $T\le90$, and study how the values of the multivariate coefficients of variation (MCV) change as time increases.
The starting positions of each particle are sampled uniformly from $[1,2]$.
In Figure \ref{fig:galton1}, we report the evolution of the five MCVs of the sampled Galtonian Trajectories, for a  sequence of equally spaced times  $T=10, 15, \dots, 85, 90$.
For completeness, we plot  in Figure \ref{fig:galton2} the simulated trajectories. 

\begin{figure}[t!]
    \centering
    \begin{minipage}{0.475\textwidth}
        \centering
        \includegraphics[width=\textwidth]{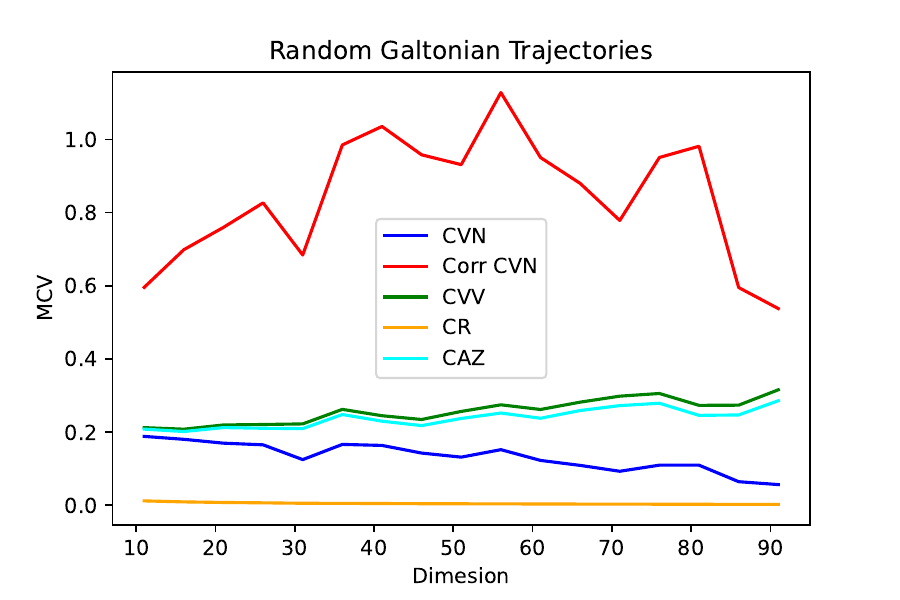}
        \caption{Multivariate Coefficients of Variation for the Galtonian Trajectories.}
        \label{fig:galton1}
    \end{minipage}\hfill
    \begin{minipage}{0.475\textwidth}
        \centering
        \includegraphics[width=\textwidth]{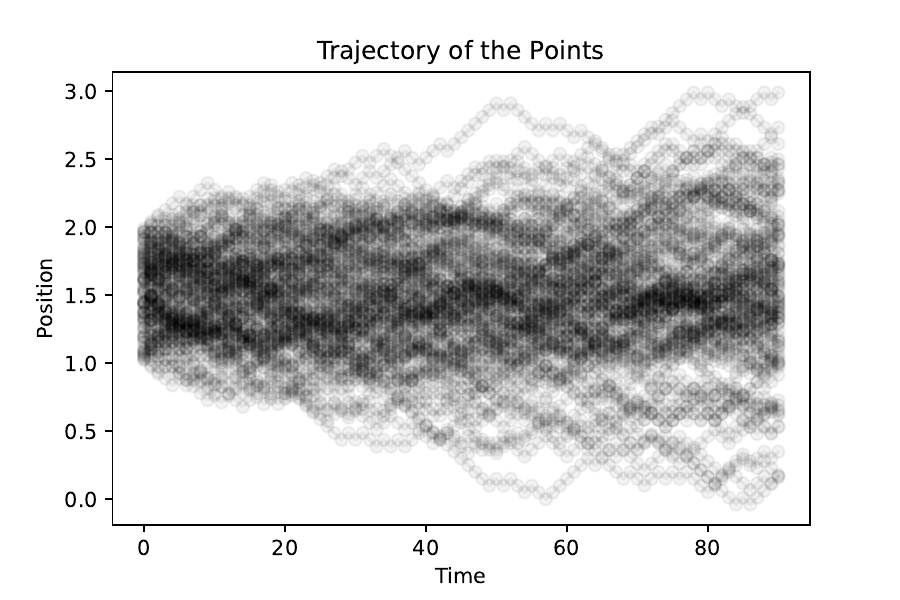}
        \caption{Visualization of the Galtonian Trajectories.}
        \label{fig:galton2}
    \end{minipage}
\end{figure}

Figure \ref{fig:galton1} shows that $\gamma_{R}$ and $\gamma_{VN}$ converge to zero as $T$ increases. Although this is in line with the theoretical results presented in \ref{sec:ourCV} and \ref{sec:other}, it does not seem intuitive, as the initial points are different.
On the other hand, $\gamma_{VV}$ and $\gamma_{CorrVN}$, and $\gamma_{AZ}$ converge, although slowly, to a non-zero value, in line with their dimension stability.
Note that $\gamma_{AZ}$ does not converge to zero, differently from what observed for the first experiment. 
This is in line with the fact that Galton trajectories do not satisfy the requirement of Definition \ref{def:dimstable} as its entries are not independent.
%

\subsection{Computational complexity}
 To conclude, we comment on the computational complexity of the considered multivariate coefficients of variation.  
From a theoretical viewpoint, computing the trace of a matrix is much cheaper than computing its determinant or its inverse matrix, especially as the dimension of the matrix increases.
From a computational viewpoint, in our second experiment we have considered very large matrices, of dimension up to  $90$ but, however, we were able to compute all coefficients of variation in a matter of seconds, using a personal laptop and a standard Python software.
This suggests that, from a  computational viewpoint, no multivariate coefficient of variation has a significant computational advantage over the others.
%


\subsection*{Funding}
The paper is the result of a close collaboration between the three authors. The work of PG was funded by the European Union - NextGenerationEU, in the framework of the GRINS- Growing Resilient, INclusive and Sustainable (GRINS PE00000018). G.T. acknowledges partial support of  IMATI (Institute of Applied  Mathematics and Information Technologies ``Enrico Magenes'').

\bibliographystyle{alpha}
\bibliography{sample}

\end{document}